\documentclass[11pt, a4paper]{article}

\usepackage{latexsym}
\usepackage{graphicx, float}
\usepackage{amssymb,amsfonts,amsthm,amsmath,graphicx,url}

\long\def\delete#1{}

\newtheorem{theorem}{Theorem}[section]
\newtheorem{lemma}[theorem]{Lemma}
\newtheorem{corollary}[theorem]{Corollary}
\newtheorem{example}[theorem]{Example}
\newtheorem{construction}[theorem]{Construction}

\newtheorem{remark}[theorem]{Remark}

\input amssym.def
\input amssym.tex

\newcommand{\be}{\begin{equation}}
\newcommand{\ee}{\end{equation}}
\newcommand{\bea}{\begin{eqnarray}}
\newcommand{\eea}{\end{eqnarray}}
\newcommand{\bean}{\begin{eqnarray*}}
\newcommand{\eean}{\end{eqnarray*}}

\def\qed{\hfill$\Box$\vspace{11pt}}

\def\la{\langle}
\def\ra{\rangle}

\def\FFF{\Bbb F}
\def\ZZZ{\Bbb Z}

\def\FFF{\mathbb{F}}

\def\BB{{\cal B}}

\def\b0{{\bf 0}}

\def\be{{\bf e}}

\def\Om{\Omega}

\def\a{\alpha}
\def\b{\beta}

\def\l{\lambda}
\def\om{\omega}

\def\Aut{{\rm Aut}}

\def\AGL{{\rm AGL}}

\def\Hol{\mathrm{Hol}}

\def\gcd{{\rm gcd}}

\title{Resolvable Mendelsohn designs and finite Frobenius groups}
\author{D. F. Hsu\\ 
Department of Computer and Information Sciences\\
Fordham University\\
New York, NY 10023, USA\\ 
\small{\it hsu@cis.fordham.edu} \\ \\
Sanming Zhou\\ 
School of Mathematics and Statistics\\
The University of Melbourne\\
Parkville, VIC 3010, Australia\\
\small{\it sanming@unimelb.edu.au}}

\begin{document}
\openup 0.8\jot \maketitle

\begin{abstract}
We prove the existence and give constructions of a $(p(k)-1)$-fold perfect resolvable $(v, k, 1)$-Mendelsohn design for any integers $v > k \ge 2$ with $v \equiv 1 \mod k$ such that there exists a finite Frobenius group whose kernel $K$ has order $v$ and whose complement contains an element $\phi$ of order $k$, where $p(k)$ is the least prime factor of $k$. Such a design admits $K \rtimes \la \phi \ra$ as a group of automorphisms and is perfect when $k$ is a prime. As an application we prove that for any integer $v = p_{1}^{e_1} \ldots p_{t}^{e_t} \ge 3$ in prime factorization, and any prime $k$ dividing $p_{i}^{e_i} - 1$ for $1 \le i \le t$, there exists a resolvable perfect $(v, k, 1)$-Mendelsohn design that  admits a Frobenius group as a group of automorphisms. We also prove that, if $k$ is even and divides $p_{i} - 1$ for $1 \le i \le t$, then there are at least $\varphi(k)^t$ resolvable $(v, k, 1)$-Mendelsohn designs that admit a Frobenius group as a group of automorphisms, where $\varphi$ is Euler's totient function.
  

\end{abstract}

\section{Introduction}
\label{sec:int}

The purpose of this paper is to explore connections between resolvable Mendelsohn designs and Frobenius groups. We will show that Frobenius groups provide a natural means for constructing resolvable Mendelsohn designs with $\l = 1$. 

All sets and groups considered in the paper are finite. Let $v \ge k \ge 2$ and $\l \ge 1$ be integers. A \emph{$(v, k, \l)$-Mendelsohn design} \cite{Mend}, or a {\em $(v, k, \l)$}-MD for short, consists of a set $X$ (of {\em points}) of cardinality $v$ and a collection $\cal B$ of cyclically ordered subsets of $X$ (called {\em blocks}) each with cardinality $k$, such that every ordered pair of elements of $X$ are consecutive in exactly $\l$ blocks. In a block $(a_1, a_2, \ldots, a_k)$ with cyclic order $a_1 < a_2 < \ldots < a_k < a_1$, $a_i$ and $a_{i+t}$ are said to be {\em $t$-apart} for $i=1, \ldots, k$ with subscripts modulo $k$. A $(v, k, \l)$-MD $(X, {\cal B})$ is called \emph{$\ell$-fold perfect} \cite{HK1} if, for $t=1, \ldots, \ell$, every ordered pair of elements of $X$ appears $t$-apart in exactly $\l$ blocks. A $(v, k, \l)$-MD is said to be \emph{perfect} \cite{Mend}, denoted by {\em $(v, k, \l)$}-PMD, if it is $(k-1)$-fold perfect. It is not difficult to see that any $(v, k, \l)$-MD has $\l v(v-1)/k$ blocks and thus satisfies $\l v(v-1) \equiv 0 \mod k$. A $(v, k, \l)$-MD is called \emph{resolvable} \cite{Ben-Men}, denoted {\em $(v, k, \l)$}-RMD, if either $v \equiv 0 \mod k$ and the set of blocks can be partitioned into $\l (v-1)$ parts each containing $v/k$ pairwise disjoint blocks, or $v \equiv 1 \mod k$ and the set of blocks can be partitioned into $\l v$ parts each containing $(v-1)/k$ pairwise disjoint blocks. We denote a resolvable perfect $(v, k, \l)$-MD by $(v, k, \l)$-RPMD. A $(v, k, \l)$-MD can be equivalently defined as a decomposition of $\l \overrightarrow{K}_v$ into edge-disjoint directed cycles of length $k$, where $\l \overrightarrow{K}_v$ is the directed complete multigraph of $v$ vertices with $\l$ directed edges between each ordered pair of vertices. From this viewpoint a $(v, k, \l)$-MD is also called a \emph{balanced directed cycle design} with parameters $(v, k, \l)$ or a \emph{$(v, k, \l)\overrightarrow{C}_k$-design} (see e.g. \cite{NP, P}). 

Mendelsohn designs are very well studied, and a large number of results about them have been produced since the 1970s. Here we are able to mention only a few results on PMDs and RMDs due to limited space. As mentioned above, a necessary condition for the existence of a $(v, k, \l)$-MD is that $\l v(v-1) \equiv 0 \mod k$. This condition has been proved to be sufficient for the existence of a $(v, k, \l)$-PMD when: (i) $k=3$, except for the non-existing $(6, 3, 1)$-PMD \cite{Ben2, Mend1}; (ii) $k=4$, except for $v=4$ and $\l$ odd, $v=8$ and $\l = 1$ \cite[Theorem 1.2]{Abel-Ben-Zh}; (iii) $k=5$, except for $\l=1$, $v \in \{6, 10\}$, and possibly for $\l=1$ and $v \in \{15, 20\}$ \cite[Theorem 1.3]{Abel-Ben-Zh}. See \cite{Abel-Ben, Abel-Ben-Zh, Ben, Miao-Zhu} for more results on PMDs.

In \cite{Ben-Men} it was proved that a $(v, k, 1)$-RMD exists if there is an algebra of order $v$ in a certain quasi-variety, that a $(q, k, 1)$-RPMD exists whenever $q$ is a prime power with $q \equiv 1 \mod k$, and that a $(v, k, 1)$-RPMD exists for sufficiently large $v$ with $v \equiv 1 \mod k$. A $(v,3,1)$-RMD exists \cite{Ben-So, BGS} if and only if $v \equiv 0$ or $1 \mod 3$ and $v \ne 6$. A $(v,4,1)$-RMD exists \cite{Ben-Zh} if and only if $v \equiv 0, 1 \mod 4$ and $v \ne 4$ except possibly when $v = 12$. A $(v,4,1)$-RPMD exists \cite{Zh1} for all $v \equiv 0 \mod 4$ other than $v=4,8$ with at most 27 possible exceptions \cite{Zh}. A $(v, k, 1)$-RPMD exists \cite{Zh1} for all sufficiently large $v$ with $v \equiv 0 \mod k$. A $(v, 5, 1)$-RPMD with $v \equiv 1 \mod 5$ exists \cite{Abel-Ben-Ge} for all $v > 6$ except possibly $v=26$, and a $(v, 5, 1)$-RPMD with $v \equiv 0 \mod 5$ exists \cite{Abel-Ben-Ge1} for all $v \ge 215$ with two known exceptions plus at most 17 possible exceptions below this value. In \cite{Zh2} it was proved that for $\l > 1$ the necessary condition $v \equiv 0, 1 \mod 4$ is also sufficient for the existence of a $(v,4,\l)$-RPMD with the exception when $v=4$ and $\l$ is odd. 

An {\em automorphism} of an MD is a permutation of its point set that permutes its blocks among themselves. The (full) {\em automorphism group} of an MD is the group of all its automorphisms with operation the usual composition of permutations. In general, for a group $G$, an MD $(X, \BB)$ is said to {\em admit} $G$ as a group of automorphisms if $G$ acts (not necessarily faithfully) on the point set $X$ and preserves the block set $\BB$. An MD is said to be \emph{based on} $G$ \cite{BS, BH, NP, P} if its point set can be identified with $G$ in such a way that the MD admits the left regular representation of $G$ as a group of automorphisms. A useful construction, called the `difference method', for constructing MDs based on groups was discussed in \cite{BS} and further developed in \cite{BH, NP, P}. A similar construction was also developed in \cite{HK1} using the language of generalized complete mappings (and generalized Mendelsohn designs). In \cite{BH} all MDs admitting a one-dimensional affine group $\AGL(1,q)$ as a group of automorphisms were classified, and in \cite{NP} all MDs admitting the holomorph of a cyclic group were classified. In \cite[Section 5]{P} MDs based on a group $G$ that admit as automorphisms elements of a certain subgroup of the automorphism group $\Aut(G)$ of $G$ were studied. 

In this paper we give constructions of RMDs and RPMDs with $\l = 1$ using Frobenius groups. Our research was motivated by the first author's work \cite{HK, HK1} (with A. D. Keedwell) on generalized complete mappings and the second author's work \cite{TZ-1, TZ-2, TZ-3, Z}  (partly with A. Thomson) on Frobenius circulant graphs. We refined the methods in \cite{HK, HK1} (see the first version of this paper at \url{https://arxiv.org/abs/1307.7455v1}), but later we found that our constructions are also in line with the difference method \cite{BH, NP, P} as all RMDs and RPMDs obtained in our paper are based on the kernels of Frobenius groups. It turns out that Frobenius groups are proper choices because the Mendelsohn designs obtained from them are always resolvable. As mentioned above, much is known about the existence of MDs with $\l=1$ in the literature. The benefit of our results is that they give natural constructions of some RMDs and RPMDs with $\l = 1$ having the additional property that an automorphism group is regular on the point set.  

The main results in the paper are as follows. We first prove (see Theorem \ref{thm:frob}) that, if there exists a Frobenius group $K \rtimes H$ with Frobenius kernel $K$ and complement $H$ such that $v = |K|$ and $H$ contains an element $\phi$ of order $k \ge 2$, then a $(p(k)-1)$-fold perfect $(v, k, 1)$-RMD based on $K$ exists, where $p(k)$ is the smallest prime factor of $k$. In particular, this $(v, k, 1)$-RMD is a $(v, k, 1)$-RPMD when $k$ is a prime. Moreover, such a $(v, k, 1)$-RMD can be easily constructed from $\phi$ and the action of $H$ on $K$. This result enables us to construct various point-transitive RMDs and RPMDs systematically using Frobenius groups. To illustrate this method, we will use Theorem \ref{thm:frob} and a known result \cite{Boy} on Ferrero pairs to prove the existence and give an explicit construction of a $(v, k, 1)$-RPMD, for any integer with prime factorization $v = p_1^{e_1} \ldots p_t^{e_t}$ and any prime $k$ dividing every $p_i^{e_i} - 1$ (see Theorem \ref{cor:ferrero}). We will also use Theorem \ref{thm:frob} and a recent result from network design \cite{TZ-3} to construct $(v, k, 1)$-RMDs for any $v = p_1^{e_1} \ldots p_t^{e_t}$ and any even $k$ dividing every $p_i - 1$ (see Theorem \ref{thm:circ}). 

All results obtained in this paper can be stated in terms of regular orthomorphisms or complete mappings of groups, owing to the close connections between MDs and orthomorphisms and complete mappings of groups as shown in \cite[Theorem 5.1]{HK1} (see also Lemma \ref{le:md} and Remark \ref{rem:converse}). Beginning with \cite{Mann} and motivated by the study of Latin squares, there is a long history of studying complete mappings. Hall and Paige \cite{Hall-Paige} proved that a finite group with a nontrivial, cyclic Sylow 2-subgroup does not admit complete mappings. The converse, which was a long-standing conjecture \cite{Hall-Paige}, was proved in 2009 by Wilcox \cite{Wilcox}, Evans \cite{Evens1} and Bray (see \cite{Evens1}).

\section{Notation and terminology}
\label{sec:def} 

To make the paper self-contained we collect here some basic definitions on orthomorphisms, complete mappings and Frobenius groups. Undefined group-theoretic definitions can be found in \cite{Dixon-Mortimer, Goren}. 

Let $G$ be a group. A bijection $\theta: G \rightarrow G$ is called a {\em complete mapping} \cite{Mann, Hall-Paige} of $G$ if the mapping $\hat{\theta}$ defined by $\hat{\theta}(x) = x \theta(x)$ is also a bijection, and an {\em orthomorphism} \cite{Evens4} if the mapping $\bar{\theta}$ defined by $\bar{\theta}(x) = x^{-1} \theta(x)$ is also a bijection. Thus, $\theta$ is a complete mapping if and only if $\hat{\theta}$ is an orthomorphism, and $\theta$ is an orthomorphism if and only if $\bar{\theta}$ is a complete mapping. 
Obviously, if $\theta$ is a complete mapping or orthomorphism of $G$, then so is the bijection $x \mapsto \theta(x)a$ for every fixed $a \in G$. Hence we may require $\theta$ to fix $1_G$, and in this case $\theta$ is said to be in \emph{canonical form}. As a permutation of $G$, $\theta$ can be decomposed into a product of cycles. A complete mapping or orthomorphism $\theta$ in canonical form is {\em $k$-regular}
\cite{HK}, where $k \ge 2$, if all cycles in this decomposition other than the trivial cycle $(1_G)$ have length $k$. If in addition $\theta, \theta^2, \ldots, \theta^{\ell}$ are all $k$-regular, then $\theta$ is said to be {\em $\ell$-fold perfect} \cite[Definition 5.7]{HK1}, where $\ell$ is a positive integer less than $k$ and $\theta^i$ is the composition of $\theta$ with itself $i$ times. In particular, $\theta$ is called \emph{perfect} \cite{HK1} if it is $(k-1)$-fold perfect. 

An {\em action} of a group $G$ on a set $\Om$ is a mapping $G \times \Om \rightarrow \Om, (x, \a) \mapsto x(\a)$ such that $1_G(\a) = \a$ and $x(y(\a)) = (xy)(\a)$ for $\a \in \Om$ and $x, y \in G$, where $1_G$ is the identity element of $G$. Call $G(\a) := \{x(\a): x \in G\}$ the {\em $G$-orbit} containing $\a$, and $G_{\a} := \{x \in G: x(\a) = \a\}$ the {\em stabilizer} of $\a$ in $G$. We say that $G$ is {\em semiregular} on $\Om$ if $G_{\a} = \{1_G\}$ for all $\a \in \Om$, {\em transitive} on $\Om$ if $G(\a) = \Om$ for some (and hence all) $\a \in \Om$, and {\em regular} on $\Om$ if it is both transitive and semiregular on $\Om$. 

If a group $H$ acts on a group $K$ such that $x(uw) = x(u) x(w)$ for $x \in H$ and $u, w \in K$, then $H$ acts on $K$ as a group. This is equivalent to saying that the mapping defined by $x \mapsto \psi_x$ is a homomorphism from $H$ to $\Aut(K)$, where $\psi_x \in \Aut(K)$ is defined by $\psi_x(u) = x(u)$. In this case the semidirect product \cite{Dixon-Mortimer} of $K$ by $H$ with respect to the action, denoted by $K \rtimes H$, is the group whose elements are ordered pairs $(u, x)$, $u \in K$, $x \in H$, with operation defined by $(u_1, x_1)(u_2, x_2) = (u_1 x_1(u_2), x_1 x_2)$. If in addition $H$ is semiregular on $K \setminus \{1_K\}$, then $K \rtimes H$ is called a {\em Frobenius group} \cite{Dixon-Mortimer, Goren}. It is well known \cite{Dixon-Mortimer, Goren} that, for a finite Frobenius group $G = K \rtimes H$, the group $K$ is a nilpotent normal subgroup of $G$ called the \emph{Frobenius kernel} of $G$, and $|H|$ is a divisor of $|K| - 1$. Here $H$ is called a \emph{Frobenius complement} \cite{Dixon-Mortimer, Goren} of $K$ in $G$. 

We can also define a Frobenius group as a transitive group $G$ on a set $\Om$ that is not regular but has the property that $1_G$ is the only element of $G$ that fixes two points of $\Om$. The Frobenius kernel $K$ of $G$ then consists of $1_G$ and the elements of $G$ fixing no point of $\Om$, and the stabilizer $H$ in $G$ of a point of $\Om$ is a complement of $G$ (see e.g.~\cite[pp.86]{Dixon-Mortimer}). Since $K$ is regular on $\Om$, we may identify $\Om$ with $K$ in such a way that $K$ acts on itself by right multiplication, and we may choose $H$ to be the
stabilizer of the identity element of $K$ so that $H$ acts on $K$ by conjugation.   

If $K$ is a nontrivial group and $H$ a nontrivial fixed-point-free subgroup of $\Aut(K)$, then $K \rtimes H$ (with respect to the natural action of $H$ on $K$) is a Frobenius group with kernel $K$ and complement $H$. (An element $\phi \in \Aut(G)$ is fixed-point-free if $\phi(x) \ne x$ for every $x \in G \setminus \{1_G\}$, and a subgroup $H$ of $\Aut(G)$ is fixed-point-free if every non-identity element of $H$ is fixed-point-free.) When the operation of $K$ is written additively (but $K$ is not necessarily abelian), such a pair $(K, H)$ is called a {\em Ferrero pair} \cite{Boy} in the literature.

The {\em left regular representation} \cite{Goren} of a group $G$ is the permutation group $L(G) = \{\l(g): g \in G\}$, where $\l(g)$ is the permutation of $G$ defined by $\l(g): x \mapsto gx$, $x \in G$. This group acts regularly on $G$ in the obvious way and is isomorphic to $G$ when $\l(g)$ is identified with $g$.

\section{Resolvable Mendelsohn designs and regular orthomorphisms}
\label{sec:frob}

We use $p(k)$ to denote the smallest prime factor of an integer $k \ge 2$. As usual, for an element $\phi$ of a group, $\langle \phi \rangle$ denotes the cyclic subgroup generated by $\phi$. 

\begin{theorem}
\label{thm:frob} 
Let $v \ge 3$ and $k \ge 2$ be integers with $v \equiv 1 \mod k$ such that there exists a Frobenius group $K \rtimes H$ with $|K| = v$ and $H$ containing an element $\phi$ of order $k$.
Then a $(p(k)-1)$-fold perfect $(v, k, 1)$-RMD exists and can be constructed based on $K$. Moreover, this $(v, k, 1)$-RMD admits $K \rtimes \la \phi \ra$ as a group of automorphisms. In particular, if $k$ is a prime, then it is a $(v, k, 1)$-RPMD. Furthermore, the $(v, k, 1)$-RMDs constructed by using conjugate elements of $H$ with order $k$ are isomorphic to each other. 
\end{theorem}

In general, unfortunately, we do not know when two non-conjugate elements of $H$ with the same order produce isomorphic $(v, k, 1)$-RMDs.   

To prove Theorem \ref{thm:frob} we need the following two lemmas. 
 
\begin{lemma}
\label{thm:froben}
Let $K \rtimes H$ be a Frobenius group. Let $\phi$ be a non-identity element of $H$ with order $k$. Then $\phi$ gives rise to a $(p(k)-1)$-fold perfect $k$-regular orthomorphism of $K$ in canonical form. Moreover, $\phi$ gives rise to a perfect $k$-regular orthomorphism of $K$ if and only if $k$ is a prime.
\end{lemma}

\begin{proof}
Since $K \rtimes H$ is a Frobenius group, $H$ is semiregular on $K \setminus \{1\}$, where $1$ is the identity element of $K$. That is, for any $x \in K$, ${\phi}(x) = x$ implies $x = 1$. Thus, for $x, y \in K$, we have: $x^{-1} {\phi}(x) =  y^{-1} {\phi}(y)$ $\Leftrightarrow$ ${\phi}(x) {\phi}(y)^{-1} =  xy^{-1}$ $\Leftrightarrow$ ${\phi}(x) {\phi}(y^{-1}) =  xy^{-1}$ $\Leftrightarrow$ ${\phi}(xy^{-1}) = xy^{-1}$ $\Leftrightarrow$ $xy^{-1} = 1$ $\Leftrightarrow$ $x = y$. In other words, the mapping from $K$ to $K$ defined by $\bar{\phi}(x) = x^{-1} {\phi}(x)$ is injective and so must be bijective as $K$ is finite. Note that ${\phi}$ fixes $1$. Therefore, ${\phi}$ is an orthomorphism of $K$ in canonical form. 

Since ${\phi}$ has order $k$, we have ${\phi}^j \ne 1_{H}$ for $1 \le j \le k-1$. Since $H$ is semiregular on $K \setminus \{1\}$, it follows that ${\phi}^{j}(x) \ne x$ for every $x \in K \setminus \{1\}$. Thus $(x, {\phi}(x), {\phi}^2(x), \ldots, {\phi}^{k-1}(x))$ is a cycle in the decomposition of the permutation ${\phi}$ of $K$ into disjoint cycles. This implies that every nontrivial cycle in the cycle decomposition of ${\phi}$ has length $k$. Therefore, ${\phi}$ is a $k$-regular orthomorphism of $K$. 

We prove further that ${\phi}$ is $(p(k)-1)$-fold perfect. Fix $i$ with $1 \le i \le p(k)-1$. Since $p(k)$ is the smallest prime factor of $k$, $i$ and $k$ are coprime. Hence the order of ${\phi}^i$ is equal to $k$. It follows from what we proved above that ${\phi}^i$ is a $k$-regular orthomorphism of $K$. Since this holds for $i = 1, \ldots, p(k)-1$, we conclude that ${\phi}$ is a $(p(k)-1)$-fold perfect $k$-regular orthomorphism of $K$. 

In the special case when $k$ is a prime, we have $p(k) = k$ and therefore ${\phi}$ is a perfect $k$-regular orthomorphism of $K$. Conversely, suppose ${\phi}$ is a perfect $k$-regular orthomorphism of $K$. Then $\phi, \phi^2, \ldots, \phi^{k-1}$ are all $k$-regular orthomorphisms of $G$. Thus, for any $1 \le i, j \le k-1$ and any $x \in K \setminus \{1\}$, we have $\phi^{ij}(x) \ne x$ and so $ij$ cannot be a multiple of $k$. Therefore, $k$ must be a prime. 
\qed
\end{proof} 
 
The following result is similar to its counterpart for complete mappings (see Theorems 5.1, 5.3 and 5.4 in \cite{HK1}), and the proof is similar to that of \cite[Theorems 5.1]{HK1}. We give its proof for the completeness of the paper. 

\begin{lemma}
\label{le:md}
Let $G$ be a group of order $v$ that admits an $\ell$-fold perfect $k$-regular orthomorphism $\theta$ (in canonical form) with cycle decomposition $(g_{11}, g_{12}, \ldots, g_{1k}) \ldots (g_{r1}, g_{r2}, \ldots, g_{rk})$, where $k \ge 2$ and $rk = v-1$. Let 
$$
{\cal B} = \cup_{g \in G} {\cal B}_g,
$$
where 
$$
{\cal B}_g = \{(gg_{11},\, gg_{12},\, \ldots,\, gg_{1k}), \ldots, (gg_{r1},\, gg_{r2},\, \ldots,\, gg_{rk})\}
$$
with each block $(gg_{i1},\, gg_{i2},\, \ldots,\, gg_{ik})$ equipped with the cyclic order $gg_{i1} < gg_{i2} < \cdots < gg_{ik} < gg_{i1}$. Then $(G, {\cal B})$ is an $\ell$-fold perfect $(v, k, 1)$-RMD with point set $G$. Moreover, $(G, {\cal B})$ admits the left regular representation of $G$ as a group of automorphisms. 
\end{lemma}

\begin{proof}
By our assumption, $\theta$ is defined by $\theta(1_G) = 1_G$ and $\theta(g_{ij}) = g_{i, j+1}$, $1 \le i \le r, 1 \le j \le k$, with the second subscripts modulo $k$. In general, for $1 \le t \le \ell$, $\theta^t(1) = 1$ and $\theta^t(g_{ij}) = g_{i, j+t}$. 

Fix $t$ with $1 \le t \le \ell$. Let $(x, y)$ be an arbitrary pair of distinct elements of $G$, so that $x^{-1} y \ne 1$. We have $(x, y) = (g g_{ij}, g g_{i, j+t})$ if and only if $x^{-1} y = g_{ij}^{-1} g_{i, j+t}$ and $g = xg_{ij}^{-1}$. However, since by our assumption $\theta^t$ is an orthomorphism of $G$, there is exactly one pair $(i, j)$ satisfying the first equation. Therefore, the ordered pair of elements $(x, y)$ are $t$-apart in exactly one block of $\cal B$. Since this holds for any $(x, y)$ and any $t$ with $1 \le t \le \ell$, it follows that $(G, {\cal B})$ is an $\ell$-fold perfect $(v, k, 1)$-MD with point set $G$. Moreover, $(G, {\cal B})$ is resolvable because of the obvious partition $\{{\cal B}_g: g \in G\}$ of $\cal B$ into $v$ parts each containing $(v-1)/k$ pairwise disjoint blocks. Clearly, $(G, {\cal B})$ admits the left regular representation of $G$ as a group of automorphisms. 
\qed
\end{proof}

\begin{remark}
\label{rem:converse}
{\em (a) Similar to \cite[Theorem 5.1]{HK1}, the converse of Lemma \ref{le:md} is also true. That is, from an $\ell$-fold perfect $(v, k, 1)$-RMD on $G$ that admits $L(G)$ as a group of automorphisms we can recover an $\ell$-fold perfect $k$-regular orthomorphism. 

(b) In Lemma \ref{le:md}, ${\cal B}_1 =  \{(g_{11},\, g_{12},\, \ldots,\, g_{1k}), \ldots, (g_{r1},\, g_{r2},\, \ldots,\, g_{rk})\}$ is a {\em basis}, and its blocks are {\em base blocks}, of the $(v, k, 1)$-RMD in the sense that all other blocks are obtained from them by applying $L(G)$. This fact is usually expressed \cite{Ben} as ${\cal B} = dev ({\cal B}_1)$.}
\end{remark}

\begin{proof}{\bf of Theorem \ref{thm:frob}}~
Let $v, k, K \rtimes H$ and $\phi$ be as in Theorem \ref{thm:frob}. 
By Lemma \ref{thm:froben}, $\phi$ gives rise to a $(p(k)-1)$-fold perfect $k$-regular orthomorphism of $K$ in canonical form, whose cycles in the cycle decomposition are $B(x) = (x, {\phi}(x), {\phi}^2(x), \ldots, {\phi}^{k-1}(x))$, $x \in K \setminus \{1\}$, where $1$ is the identity element of $K$. Note that $B(x) = B(y)$ as sets if and only if $y =\phi^i(x)$ for some $i$. Regard $B(x)$ as a block with cyclic order $x < {\phi}(x) < {\phi}^2(x) < \cdots < {\phi}^{k-1}(x) < x$, and define ${\cal B}_1 (\phi) = \{B(x): x \in K \setminus \{1\}\}$ with duplicated blocks counted only once. More explicitly, letting the $\la \phi \ra$-orbits on $K \setminus \{1\}$  be $\la \phi \ra(x_1), \ldots, \la \phi \ra(x_r)$, where $x_1, \ldots, x_r \in K \setminus \{1\}$ and $r = (v-1)/k$, we have 
$$
{\cal B}_1 (\phi) = \{B(x_1), \ldots, B(x_r)\}. 
$$
Define 
$$
{\cal B}(\phi) = dev ({\cal B}_1(\phi)) = \cup_{g \in K} {\cal B}_g (\phi),
$$
where 
$$
{\cal B}_g (\phi) = g {\cal B}_1 (\phi) = \{gB(x_1), \ldots, gB(x_r)\},
$$ 
where
$$
gB(x_i) = (gx_i, g{\phi}(x_i), g{\phi}^2(x_i), \ldots, g{\phi}^{k-1}(x_i)),\;\, i = 1, \ldots, r.
$$ 
By Lemma \ref{le:md}, $(K, {\cal B}(\phi))$ is a $(p(k)-1)$-fold perfect $(v, k, 1)$-RMD. Since $\la \phi \ra$ leaves each block of ${\cal B}_1(\phi)$ invariant and permutes its elements cyclically, one can verify that $(K, {\cal B}(\phi))$ admits $\la \phi \ra$ as a group of automorphisms. Since by Lemma \ref{le:md}, $(K, {\cal B}(\phi))$ also admits $L(K)$ as a group of automorphisms, it admits $K \rtimes \la \phi \ra$ as a group of automorphisms (with $K$ identified with $L(K)$).   

If $k$ is a prime, then by Lemma \ref{thm:froben}, $\phi$ gives rise to a perfect $k$-regular orthomorphism of $K$, and hence $(K, {\cal B}(\phi))$ is a $(v, k, 1)$-RPMD by Lemma \ref{le:md}. 

Finally, assume that $\phi' = \psi^{-1} \phi \psi$ is a conjugate of $\phi$, where $\psi \in H$. Then a typical base block in ${\cal B}_1 (\phi')$ is $(x, \psi^{-1} \phi \psi(x), \ldots, \psi^{-1} \phi^{k-1} \psi (x))$, which can be expressed as 
$$
(\psi^{-1} (y), \psi^{-1} \phi (y), \ldots, \psi^{-1} \phi^{k-1} (y)) = \psi^{-1} (y, \phi (y), \ldots, \phi^{k-1} (y)),
$$ 
where we set $y = \psi(x)$. From this one can verify that $\psi$ (viewed as a bijection of $K$) is an isomorphism from $(K, {\cal B}(\phi))$ to $(K, {\cal B}(\phi'))$. 
\qed
\end{proof}

In Theorem \ref{thm:frob} and Lemma \ref{thm:froben} we essentially dealt with the Frobenius group $K \rtimes \langle \phi \rangle$. Thus we have the following corollary of Theorem \ref{thm:frob}.

\begin{corollary}
\label{cor:frob}
Let $v \ge 3$ and $k \ge 2$ be integers such that there exists a group $K$ with order $v$ that admits a fixed-point-free automorphism $\phi$ of order $k$. Then there exists a $(p(k)-1)$-fold perfect $(v, k, 1)$-RMD based on $K$ that admits $K \rtimes \la \phi \ra$ as a group of automorphisms. If in addition $k$ is a prime, then this $(v, k, 1)$-RMD is a $(v, k, 1)$-RPMD.
\end{corollary}

The assumption in this corollary (as well as the next one) implies $v \equiv 1 \mod k$. The well-known Cauchy's theorem in group theory asserts that for any group $H$ and any prime divisor $k$ of $|H|$, $H$ contains an element of order $k$. Combining this with Theorem \ref{thm:frob} we obtain the following result. 

\begin{corollary}
\label{cor:frob1}
Let $v \ge 3$ be an integer and $k \ge 2$ a prime such that there exists a Frobenius group $K \rtimes H$ with $|K| = v$ and with $k$ dividing $|H|$. Then there exists a $(v, k, 1)$-RPMD based on $K$ that admits $K \rtimes \la \phi \ra$ as a group of automorphisms, where $\phi$ is an element of $H$ with order $k$.
\end{corollary}

Applying Theorem \ref{thm:frob}, Corollary \ref{cor:frob} or Corollary \ref{cor:frob1} to various Frobenius groups, we can obtain point-transitive Mendelsohn designs of specific parameters.  
In the special case when the group is $\AGL(1, q)$, we obtain the following well-known construction (see e.g. \cite{Ben, Miao-Zhu, Yin}) which is included for illustration only. (A related result that  generalizes \cite[Theorem 2.3]{BH} is the classification \cite[Theorem 5.4]{P} of all $(q, k, 1)$-MDs based on $(\FFF_q, +)$ admitting a subgroup of $(\FFF^*_q, \cdot)$ as a group of automorphisms.)

\begin{example}
\label{ex:classic}
{\em
$\AGL(1, q)$ consists of all affine transformations $t_{\a, \b}: \xi \mapsto \a \xi + \b$ of $\FFF_q$, $\a \in \FFF_q^*, \b \in \FFF_q$, where $q$ is a prime power. It is well known that $\AGL(1, q) = K \rtimes H$ is sharply 2-transitive \cite{Dixon-Mortimer} on $\FFF_q$ and hence is a Frobenius group, where $K = \{t_{1, \b}: \b \in \FFF_q\} \cong (\FFF_q, +)$ and $H = \{t_{\a, 0}: \a \in \FFF^*_q\} \cong (\FFF^*_q, \cdot)$. Since $H$ is a cyclic group of order $q-1$, for every divisor $k \ge 2$ of $q-1$, $H$ has a unique element $\phi = t_{\a,0}$ of order $k$, where $\a = \om^{r}$ with $\om$ a primitive element of $\FFF_q$ and $r = (q-1)/k$. Moreover, $\la \phi \ra$ is isomorphic to the subgroup $\la \a \ra$ of $\FFF_q^*$, and so we may identify these two cyclic groups. By Theorem  \ref{thm:frob}, a $(q, k, 1)$-RMD $(K, {\cal B} (\phi))$ exists and can be constructed explicitly. A typical base block in ${\cal B}_1(\phi)$ is of the form $(x, x \a, \ldots, x \a^{k-1})$. Hence ${\cal B}_1 (\phi) = \{(1, \om^r, \ldots, \om^{r(k-1)}), \ldots, (\om^{r-1}, \om^{2r-1}, \ldots, \om^{r(k-1)+(r-1)})\}$ and so the blocks of ${\cal B} (\phi)$ are 
$$
(\om^i, \om^{i+r}, \ldots, \om^{i+r(k-1)}), \ldots, (\om^{i+r-1}, \om^{i+2r-1}, \ldots, \om^{i+r(k-1)+(r-1)}), \; i = 0, 1, \ldots, r.
$$ 
In the case when $k$ is a prime factor of $q-1$, $(K, {\cal B} (\phi))$ is a $(q, k, 1)$-RPMD by Theorem \ref{thm:frob}. 
}
\end{example}

In \cite[Theorem 1]{Boy}, Boykett proved the following: Let $v = p_{1}^{e_1} \ldots p_{t}^{e_t}$ (prime factorization) and $k$ be positive integers. Then there exists a Ferrero pair $(K, H)$ such that $|K| = v$ and $|H| = k$ if and only if $k$ divides $p_{i}^{e_i} - 1$ for $1 \le i \le t$. Using this, we obtain the following special case of Theorem \ref{thm:frob}, where $\gcd$ stands for the greatest common divisor. 

\begin{theorem}
\label{cor:ferrero}
Let $v = p_{1}^{e_1} \ldots p_{t}^{e_t} \ge 3$ be an integer in prime factorization, and let $k$ be a prime factor of $\gcd(p_{1}^{e_1} - 1, \ldots, p_{t}^{e_t} - 1)$. Then there exists a $(v, k, 1)$-RPMD  that admits a Frobenius group $K \rtimes H$ with $|K|=v$ and $|H|=k$ as a group of automorphisms.  
\end{theorem}

\begin{proof}
Since $k$ divides $\gcd(p_{1}^{e_1} - 1, \ldots, p_{t}^{e_t} - 1)$, by the result of Boykett mentioned above there exists a Ferrero pair $(K, H)$ with $|K| = v$ and $|H| = k$. Since $k$ is a prime, $H$ is a cyclic group and every non-identity element of it has order $k$. The result follows from Theorem \ref{thm:frob}.
\qed
\end{proof}

Theorem \ref{cor:ferrero} is in the same spirit as the following known results: 
\begin{itemize}
\item[(i)] a $(v, k, 1)$-RPMD exists \cite{Ben-Men} for any prime power $v$ and any divisor $k \ge 2$ of $v-1$; 
\item[(ii)] a $(v, k, 1)$-PMD exists \cite[Corollary 2.5]{Mend} for any integer $v = p_{1}^{e_1} \ldots p_{t}^{e_t} \ge 3$ and any divisor $k \ge 2$ of $\gcd(p_{1}^{e_1} - 1, \ldots, p_{t}^{e_t} - 1)$. 
\end{itemize}
Nevertheless, we remark that neither of these results implies Theorem \ref{cor:ferrero}. 

\begin{construction}
\label{con:ferrero}
{\em Following the line of the proof of Theorem \ref{thm:frob}, we can construct the $(v, k, 1)$-RPMD in Theorem \ref{cor:ferrero} explicitly. Denote $q_i = p_{i}^{e_i}$ and let $K = \oplus_{i=1}^t \FFF_{q_i}$ be the direct sum of the additive groups of the finite fields $\FFF_{q_i}$. Since $k$ divides $q_i - 1$, $r_i = (q_i - 1)/k$ is an integer for each $i$. Let $\om_i$ be a primitive element of $\FFF_{q_i}$ and $H_i = \la \om_{i}^{r_i} \ra$ be the subgroup of $\FFF^*_{q_i}$ with order $k$. We denote a copy of $H_1 \cong \cdots \cong H_t$ by $H = \la \om \ra$. Then $H$ acts on $K$ by $\om^j (x_1, \ldots, x_t) = (\om_1^{j r_1} x_1, \ldots, \om_t^{j r_t} x_t)$, for $(x_1, \ldots, x_t) \in K$ and $\om^j \in H$. It can be verified that $H$ acts fixed-point-freely on $K$ as a group, and $H$ is isomorphic to a subgroup of $\Aut(K)$. Thus $(K, H)$ is a Ferrero pair with $|K| = v$ and $|H| = k$. Since $k$ is a prime, every non-identity element of $H$ has order $k$. For $i = 1, \ldots, k-1$, the $(v, k, 1)$-RPMD $(K, {\cal B}(\om^i))$ obtained from $\om^i$ is as follows: the basis ${\cal B}_{\bf 0}(\om)$ consists of blocks of the form $B(x_1, \ldots, x_t) = ((x_1, \ldots, x_t), (\om_1^{i r_1} x_1, \ldots, \om_t^{i r_t} x_t), \ldots, (\om_1^{(k-1)i r_1} x_1, \ldots, \om_t^{(k-1)i r_t} x_t))$, $(x_1, \ldots, x_t) \in K \setminus \{(0, \ldots, 0)\}$. We have ${\cal B}(\om) = \cup_{(y_1, \ldots, y_t) \in K} ({\cal B}_{\bf 0}(\om) + (y_1, \ldots, y_t))$, where ${\cal B}_{\bf 0}(\om) + (y_1, \ldots, y_t)$ consists of all blocks $B(x_1, \ldots, x_t) + (y_1, \ldots, y_t) = ((x_1 + y_1, \ldots, x_t + y_t), (\om_1^{i r_1} x_1 + y_1, \ldots, \om_t^{i r_t} x_t + y_t), \ldots, (\om_1^{(k-1)i r_1} x_1 + y_1, \ldots, \om_t^{(k-1)i r_t} x_t + y_t))$.}
\end{construction}

There has been extensive research on the existence of $(v, k, \l)$-PMDs for a fixed (especially small) integer $k$ (see \cite{Ben} for a survey). Theorem \ref{cor:ferrero} asserts that, for a fixed prime $k$, a $(p_{1}^{e_1} \ldots p_{t}^{e_t}, k, 1)$-RPMD exists for any prime powers $p_{i}^{e_i}$ as long as $k$ divides all $p_{i}^{e_i} - 1$, $1 \le i \le t$. In particular, for any primes $p_1, \ldots, p_t \equiv 1 \mod k$ and any integers $e_1, \ldots, e_t \ge 1$, there exists a $(p_{1}^{e_1} \ldots p_{t}^{e_t}, k, 1)$-RPMD.

\section{Constructing resolvable Mendelsohn designs from cyclic groups}
\label{sec:cyclic}

In this section we construct RMDs from cyclic groups by using Theorem \ref{thm:frob} and recent results \cite{TZ-1, TZ-2, TZ-3, Z} on first-kind Frobenius circulant graphs. As usual we use $\ZZZ_n$ to denote the additive group of integers modulo $n$ and $\ZZZ_n^* = \{[u]: 1 \le u \le n-1, \gcd(n,u) = 1\}$ the multiplicative group of units of ring $\ZZZ_n$. Then $\Aut(\ZZZ_n) \cong \ZZZ_n^*$ and $\ZZZ^*_{n}$ acts on $\ZZZ_{n}$ by the usual multiplication: $[x][u] = [xu]$, $[x] \in \ZZZ_{n}$, $[u] \in \ZZZ^*_{n}$. The semidirect product $\ZZZ_{n} \rtimes \ZZZ^*_{n}$ acts on $\ZZZ_{n}$ such that $[x]^{([y], [u])} = [(x+y)u]$ for $[x], [y] \in \ZZZ_{n}$ and $[u] \in \ZZZ^*_{n}$.  

Denote by $\varphi$ Euler's totient function. The main result in this section is as follows.  

\begin{theorem}
\label{thm:circ}
Let $v = p_1^{e_1}\ldots p_t^{e_t} \ge 3$ be an odd integer in prime factorization. Then for every even divisor $k$ of $\gcd(p_{1} - 1, \ldots, p_{t} - 1)$, there exist at least $\varphi(k)^{t}$ $(v, k, 1)$-RMDs based on $\ZZZ_v$, and each of them can be constructed from some $[a] \in \ZZZ_v^*$ of the form $a = \sum_{i=1}^t (v/p_i^{e_i}) a_i b_i$ and admits $\ZZZ_v \rtimes \la [a] \ra$ as a group of automorphisms, where $b_i$ is the inverse of $v/p_i^{e_i}$ in $\FFF_{p_i^{e_i}}$, and $a_i$ satisfies $a_i \equiv \eta_i^{m_i \varphi(p_i^{e_i})/k} \mod {p_i^{e_i}}$ for a fixed primitive element $\eta_i$ of $\FFF_{p_i^{e_i}}$ and an integer $m_i$ coprime to $k$. 
\end{theorem}

\begin{remark}
\label{re:cyclic}
{\em (a) It would be interesting to understand whether and when some of these $\varphi(k)^{t}$ $(v, k, 1)$-RMDs are isomorphic to each other. A construction of such RMDs will be given in the proof of Theorem \ref{thm:circ}. 

(b) In \cite[Corollary 2]{NP} it was shown that a $(v, k, 1)$-MD admitting the holomorph $\Hol(\ZZZ_v) \cong \ZZZ_{v} \rtimes \ZZZ^*_{v}$ of $\ZZZ_v$ as a group of automorphisms exists if and only if one of the following holds: (i) $k=2$; (ii) $p \equiv 1 \mod k$ for every prime factor $p$ of $v$; (iii) $k$ is the least prime factor of $v$, $k^2$ does not divide $v$, and $p \equiv 1 \mod k$ for every prime factor $p$ of $v$ other than $k$. Theorem \ref{thm:circ} implies that under condition (ii) many $(v, k, 1)$-RMDs based on $\ZZZ_v$ exist and can be easily constructed. None of Theorem \ref{thm:circ} and \cite[Corollary 2]{NP} is implied by the other.

(c) Similar to Lemma \ref{thm:froben}, the result in Theorem \ref{thm:circ} can be stated in terms of orthomorphisms: under the same assumption $\ZZZ_v$ admits at least $\varphi(k)^{t}$ $k$-regular orthomorphisms. D. F. Hsu conjectured that, for every odd integer $n \ge 3$ and every divisor $k$ of $n-1$, a $k$-regular complete mapping of $\ZZZ_n$ exists, or equivalently a $k$-regular orthomorphism exists. This was confirmed when $k =2$ or $(n-1)/2$ (see \cite[Theorem 3]{FGT}). The following corollary of Theorem \ref{thm:circ} provides further support to Hsu's conjecture.}
\end{remark}

\begin{corollary}
\label{coro:hsu}
Let $n \ge 3$ be an odd integer and $k$ a divisor of $n-1$. If $k$ divides $p-1$ for every prime factor $p$ of $n$, then a $k$-regular complete mapping of $\ZZZ_n$ exists.  
\end{corollary}
    
\begin{proof}{\bf of Theorem \ref{thm:circ}}~
Let $v$ and $k$ be as in Theorem \ref{thm:circ}. Let $[a]$ be an element of $\ZZZ_v^*$ with order $k$ such that $H = \la [a] \ra$ is semiregular on $\ZZZ_v \setminus \{[0]\}$. (See below for the existence of such elements $[a]$.) Then $\ZZZ_v \rtimes H$ is a Frobenius group. It can be verified (see \cite[Lemma 4]{TZ-1}) that the semiregularity of $H$ on $\ZZZ_v \setminus \{[0]\}$ is equivalent to saying that $[a^i - 1] \in \ZZZ_v^*$ for $1 \le i \le k-1$. 

Let $H[x_1] = H$, $H[x_2]$, $\ldots$, $H[x_r]$ be the $H$-orbits on $\ZZZ_v \setminus \{[0]\}$, where we assume $[x_1] = [1]$ without loss of generality. Then each $H[x_i]$ has length $k$ and $kr = v-1$. Define $\phi_a: \ZZZ_v \rightarrow \ZZZ_v$ by $\phi_a([0]) = [0]$ and $\phi_a([a^sx_i]) = [a^{s+1}x_{i}]$, $1 \le i \le r,\, 0 \le s \le k-1$. In line with the proof of Lemma \ref{thm:froben}, since $[a-1] \in \ZZZ_v^*$ as noted above, one can verify that the mapping from $\ZZZ_v$ to itself defined by $\bar{\phi}_a([0]) = [0]$ and $\bar{\phi}_a([a^sx_i]) = \phi_a([a^sx_i]) - [a^sx_i] = [a^s (a-1) x_i]$ is bijective. Hence $\phi_a$ is a $k$-regular orthomorphism of $\ZZZ_v$. Thus, by Theorem \ref{thm:frob} and its proof, $\phi_a$ produces a $(v, k, 1)$-RMD, ${\cal B}(a)$, whose base blocks are $B_{a, i} = ([x_i], [a x_i], [a^{2}x_i], \ldots, [a^{k-1}x_i])$, for $1 \le i \le r$. More explicitly, ${\cal B}(a) = \{[x]+B_{a, i}: [x] \in \ZZZ_n, 1 \le i \le r\}$, where we set $[x] + B_{a, i} = ([x+x_i], [x+a x_i], [x+a^{2}x_i], \ldots, [x+a^{k-1}x_i])$. 
 
It was proved in \cite[Theorem 2.7]{TZ-3} that every element $[a]$ of $\ZZZ_v^*$ with order $k$ such that $\la [a] \ra$ is semiregular on $\ZZZ_v \setminus \{[0]\}$ can be constructed in the way as stated in the theorem, and vice versa. Since $k$ divides each $p_i - 1$, it divides $p_i^{e_i} - 1$ and so $\ZZZ_{p_i^{e_i}}^*$ has exactly one subgroup of order $k$. Thus, as noted in the proof of \cite[Theorem 2.7]{TZ-3}, without loss of generality we may fix the primitive element $\eta_i$ of $\FFF_{p_i^{e_i}}$. There are exactly $\varphi(k)$ values $a_i$ satisfying $a_i \equiv \eta_i^{m_i \varphi(p_i^{e_i})/k} \pmod {p_i^{e_i}}$, each corresponding to a different value of $m_i$. Since this is true for $i = 1, \ldots, t$ and since each $b_i \pmod {p_i^{e_i}}$ is unique, it follows that there are exactly $\varphi(k)^t$ different elements $[a]$ of $\ZZZ_v^*$ with order $k$ such that $\la [a] \ra$ is semiregular on $\ZZZ_v \setminus \{[0]\}$. (Note that if $[a]$ is such an element, then so is $[a^j]$ provided $\gcd(j, k) = 1$. In this case, $\la [a] \ra = \la [a^{j}] \ra$ but the $(v, k, 1)$-RMD ${\cal B}(a^j)$ is not necessarily identical to ${\cal B}(a)$.) Therefore, there exist at least $\varphi(k)^t$ $(v, k, 1)$-RMDs each with the stated properties.
\qed 
\end{proof}

The following is a corollary of Theorem \ref{thm:circ} and \cite[Theorem 2]{TZ-1}.

\begin{corollary}
\label{thm:4}
Let $v = p_1^{e_1}\ldots p_t^{e_t} \ge 5$ be an integer in prime factorization such that each $p_i \equiv 1 \mod 4$. Then there are at least $2^{t}$ $(v, 4, 1)$-RMDs based on $\ZZZ_v$, and each of them is produced by a solution $a$ to the congruence equation $x^2 + 1 \equiv 0 \mod v$ and admits $\ZZZ_v \rtimes \la [a] \ra$ as a group of automorphisms.
\end{corollary}

\begin{proof}
Since each $p_i \equiv 1 \mod 4$, $4$ is a divisor of $\gcd(p_1 -1, \ldots, p_t-1)$. Since $\varphi(4) = 2$, by Theorem \ref{thm:circ} there are at least $2^{t}$ $(v, 4, 1)$-RMDs. Moreover, each of them is constructed from an element $[a] \in \ZZZ_v^*$ of order $4$ such that $\la [a] \ra$ is semiregular on $\ZZZ_v \setminus \{[0]\}$ and admits corresponding $\ZZZ_v \rtimes \la [a] \ra$ as a group of automorphisms. It can be verified (see \cite[Theorem 2]{TZ-1}) that such elements $a$ are in one-to-one correspondence with the solutions to $x^2 + 1 \equiv 0 \mod v$. 
\qed
\end{proof}

In \cite[Theorem 2]{TZ-2} it was proved that, if every prime factor of an integer $v \ge 7$ is congruent to 1 modulo 6, then there are exactly $2^{t}$ ($= \varphi(6)^t$) elements $[a]$ of $\ZZZ_v^*$ of order 6 such that $\la [a] \ra$ is semiregular on $\ZZZ_v \setminus \{[0]\}$, and moreover they are in one-to-one correspondence with the solutions to $x^2 - x + 1 \equiv 0 \mod v$. Thus, similar to Corollary \ref{thm:4}, we obtain the following special case of Theorem \ref{thm:circ}. 

\begin{corollary}
\label{thm:6}
Let $v = p_1^{e_1}\ldots p_t^{e_t} \ge 7$ be an integer in prime factorization such that each $p_i \equiv 1 \mod 6$. Then there are at least $2^{t}$ $(v, 6, 1)$-RMDs based on $\ZZZ_v$, and each of them is produced by a solution $a$ to the congruence equation $x^2 - x + 1 \equiv 0 \mod v$ and admits $\ZZZ_v \rtimes \la [a] \ra$ as a group of automorphisms.
\end{corollary}

We conclude this paper by an example which illustrates how RMDs in Corollary \ref{thm:4} can be constructed following the proof of Theorem \ref{thm:circ}.  
 
\begin{example}
\label{ex:4}
{\em 
Consider $v=53$. Then $x=23$ is a solution to $x^2 + 1 \equiv 0 \mod 53$. Let $H = \la [23] \ra \le \ZZZ_{53}^*$. Then the $H$-orbits on $\ZZZ_{53} \setminus \{[0]\}$ are as follows: 
$$
H = \{[1], [23], [52], [30]\}, H[2] = \{[2], [46], [51], [7]\}, H[3] = \{[3], [16], [50], [37]\},  
$$
$$
H[4] = \{[4], [39], [49], [14]\}, H[24] = \{[24], [22], [29], [31]\}, H[25] = \{[25], [45], [28], [8]\}, \\
$$
$$
H[26] = \{[26], [15], [27], [38]\}, H[47] = \{[47], [21], [6], [32]\}, H[48] = \{[48], [44], [5], [9]\},
$$ 
$$
H[17] = \{[17], [20], [36], [33]\}, H[18] = \{[18], [43], [35], [10]\}, H[40] = \{[40], [19], [13], [34]\},\\
$$ 
$$
H[41] = \{[41], [42], [12], [11]\}.
$$
Thus, in view of the proof of Lemma \ref{thm:froben}, the permutation 
$$
(1\ 23\ 52\ 30) (2\ 46\ 51\ 7) \cdots (41\ 42\ 12\ 11)
$$ 
is a 4-regular orthomorphism of $\ZZZ_{53}$. In view of the proof of Theorem \ref{thm:circ}, this orthomorphism gives rise to the $(53, 4, 1)$-RMD ${\cal B}(53)$ whose basis ${\cal B}_0 (53)$ consists of 
$$
(1, 23, 52, 30), (2, 46, 51, 7), \ldots, (41, 42, 12, 11).
$$ 
Note that ${\cal B}(53) = dev ({\cal B}_0 (53)) = \cup_{i=0}^{52} ({\cal B}_0 (53) + i)$, where ${\cal B}_0 (53) + i$ is obtained from ${\cal B}_0 (53)$ by adding $i$ to each block of ${\cal B}_0 (53)$ coordinate-wise with addition modulo 53. For example, ${\cal B}_0 (53) + 1$ consists of $(2, 24, 0, 31)$, $(3, 47, 52, 8), \ldots, (42, 43, 13, 12)$. 
}
\end{example}

\medskip
\noindent {\bf Acknowledgements}~ Zhou was supported by a Future Fellowship (FT110100629) of the Australian Research Council. 

\small
{

}

\end{document}